\documentclass[reqno]{amsart}
\numberwithin{equation}{section}

\usepackage{bbm}
\usepackage{comment}
\usepackage{verbatim}
\usepackage{mathrsfs}
\usepackage{mathtools}
\usepackage{latexsym}
\usepackage{epsfig}
\usepackage{amsmath}
\usepackage[usenames,dvipsnames]{xcolor}
\usepackage{graphics}
\usepackage[utf8]{inputenc}
\usepackage{amssymb}
\usepackage{amsthm}
\usepackage[alphabetic]{amsrefs}
\usepackage{amsopn}
\usepackage{amscd}
\usepackage[all,knot]{xy}
\xyoption{all}
\usepackage{rotating}
\usepackage{tikz}
\usepackage{hyperref}
\usepackage{tikz}
\usepackage{tikz-cd}

\theoremstyle{plain}
\newtheorem{thm}{Theorem}[section]

\newtheorem{prop}[thm]{Proposition}
\newtheorem{cor}[thm]{Corollary}

\newtheorem*{thm*}{Theorem}
\newtheorem*{lem*}{Lemma}
\newtheorem*{prop*}{Proposition}
\newtheorem*{cor*}{Corollary}

\theoremstyle{definition}

\newtheorem*{defn*}{Definition}

\newtheorem{ex}[thm]{Example}
{}
\newtheorem{rem}[thm]{Remark}
\newtheorem*{rem*}{Remark}

{}
{}
{}
{}
\newtheorem{qn}[thm]{Question}{}
{}

\theoremstyle{remark}
{}
{}
{}

\def\to{\longrightarrow} 

\def\ZZ{\mathbb{Z}}
\DeclareUnicodeCharacter{221E}{$\infty$}

\def\sfD{\mathsf{D}}

\def\sfT{\mathsf{T}}

\newcommand{\mysetminusD}{\hbox{\tikz{\draw[line width=0.6pt,line cap=round] (3pt,0) -- (0,6pt);}}}
\newcommand{\mysetminusT}{\mysetminusD}
\newcommand{\mysetminusS}{\hbox{\tikz{\draw[line width=0.45pt,line cap=round] (2pt,0) -- (0,4pt);}}}
\newcommand{\mysetminusSS}{\hbox{\tikz{\draw[line width=0.4pt,line cap=round] (1.5pt,0) -- (0,3pt);}}}

\newcommand{\mysetminus}{\mathbin{\mathchoice{\mysetminusD}{\mysetminusT}{\mysetminusS}{\mysetminusSS}}}

\def\op{\mathrm{op}}

\DeclareMathOperator{\add}{add}

\DeclareMathOperator{\Hom}{Hom}

\DeclareMathOperator{\modu}{\mathsf{mod}}

\DeclareMathOperator{\smodu}{\underline{\mathsf{mod}}}

\DeclareMathOperator{\RHom}{\mathrm{R}Hom}

\DeclareMathOperator{\Ext}{Ext}

\DeclareMathOperator{\Fun}{Fun}

\DeclareMathOperator{\cone}{cone}

\DeclareMathOperator{\thick}{thick}
\DeclareMathOperator{\Thick}{Thick}

\DeclareMathOperator{\rad}{rad}

\definecolor{internationalkleinblue}{rgb}{0.0, 0.18, 0.65}

\title{Ore's theorem for thick subcategories}

\thanks{SG was supported by VILLUM FONDEN (Grant Number VIL42076), GS was supported by the Danmarks Frie Forskningsfond (grant ID: 10.46540/4283-00116B)}

\author{Sira Gratz}
\address{Sira Gratz, Aarhus University, Department of Mathematics, Ny Munkegade 118, bldg. 1530
DK-8000 Aarhus C, Denmark
}
\email{sira@math.au.dk}

\author{Greg Stevenson}
\address{Greg Stevenson, Aarhus University, Department of Mathematics, Ny Munkegade 118, bldg. 1530
DK-8000 Aarhus C, Denmark
}
\email{greg@math.au.dk}


\begin{document}

\begin{abstract}
\noindent We characterize those finite groups for which the bounded derived category of finite dimensional representations over an algebraically closed field of characteristic $p$ has distributive lattice of thick subcategories: they are precisely the $p$-nilpotent groups. Along the way we give necessary and sufficient criteria for the bounded derived category and perfect complexes of a finite dimensional $k$-algebra to have distributive lattices of thick subcategories.
\end{abstract}

\maketitle




\section{Introduction}

Given an object one can consider its collection of subobjects. In general, these subobjects form a lattice and it is natural to wonder how properties of this lattice reflect the properties of the original object. Many instances of this question have been considered over the years. For instance, in 1938 Ore proved that the collection of subgroups of a finite group $G$ is a distributive lattice precisely when $G$ is a cyclic group \cite{Ore}.

The purpose of this note is to consider what we can learn from distributivity in the derived setting. We replace the group $G$ by one of the bounded derived category $\sfD^\mathrm{b}(kG)$, the perfect complexes $\sfD^\mathrm{perf}(kG)$, or the stable category $\smodu kG$, where $k$ is an algebraically closed field of characteristic $p$, and we study the lattice of thick subcategories in each case. It turns out that these lattices are distributive more often than the lattice of subgroups, but it is still a quite restrictive condition. The analogue of Ore's theorem is:

\begin{thm*}[Corollary~\ref{pnilpotent3}]
Let $G$ be a finite group and $k$ an algebraically closed field of characteristic $p$. Then the following are equivalent:
\begin{itemize}
\item[(1)] $G$ is $p$-nilpotent;
\item[(2)] The lattice of thick subcategories of $\sfD^\mathrm{b}(kG)$ is distributive;
\item[(3)] The lattice of thick subcategories of $\sfD^\mathrm{perf}(kG)$ is distributive;
\end{itemize}
\end{thm*}

The situation for $\smodu kG$ is more involved and we give some examples and prove that the lattice of thick subcategories is distributive in several examples which are not $p$-nilpotent.

The theorem follows from an analysis of distributivity for finite dimensional algebras. This turns out to be an extremely restrictive characterization as the following theorem, which summarises the results of Corollaries~\ref{cor:dist1} and \ref{cor:dist2}, shows.

\begin{thm*}
Let $A$ be a finite dimensional $k$-algebra. Then the lattice of thick subcategories of $\sfD^\mathrm{b}(A)$ (respectively $\sfD^\mathrm{perf}(A)$) is distributive if and only if $A$ is Morita equivalent to a product of local algebras $A_i$ such that each $\sfD^\mathrm{b}(A_i)$ (respectively $\sfD^\mathrm{perf}(A_i)$) has a distributive lattice of thick subcategories.
\end{thm*}

The article also includes a number of examples and we pose several questions.



\section{Notation}

For a triangulated category $\sfT$ we denote by $\Thick(\sfT)$ the complete lattice of thick subcategories of $\sfT$. The meet $\wedge$ is given by intersection and the join $\vee$ by the smallest thick subcategory containing the union of the thick subcategories in question. For $X\in \sfT$ we let $\thick(X)$ be the smallest thick subcategory containing $X$. Further details can be found in various places, for instance \cite{SiraGreg}.

Throughout we work over an algebraically closed field $k$ (often of characteristic $p$ where $p$ divides the order of some finite group $G$). We denote by $A$ a finite dimensional $k$-algebra which, when convenient, we assume is basic as we are only concerned with properties that are Morita-invariant. Thus $A$ can be represented as a quotient of a path algebra. We denote by $J$ the radical $\rad(A)$. We use $\sfD^\mathrm{b}(A)$ to denote the bounded derived category of finite dimensional $A$-modules and $\sfD^\mathrm{perf}(A)$ to denote the thick subcategory of perfect complexes.






\section{Distributivity for the bounded derived category}

We start by showing that distributivity of the lattice of thick subcategories for the bounded derived category is a strong constraint.

\begin{thm}\label{thm:dist1}
If $A$ is finite dimensional $k$-algebra and $\Thick(\sfD^\mathrm{b}(A))$ is distributive then $A$ is Morita equivalent to a product of local rings $A_i$ such that each $\Thick(\sfD^\mathrm{b}(A_i))$ is distributive.
\end{thm}
\begin{proof}
Without loss of generality $A$ is basic with simple modules $S_i$ and corresponding indecomposable projectives $P_i$ for $1\leq i \leq n$. We always have
\[
\sfD^\mathrm{b}(A) = \thick(S_1,\ldots,S_n) = \bigvee_{i=1}^n \thick(S_i).
\]
So for each $1\leq j \leq n$ distributivity yields that
\begin{align*}
\thick(P_j) &= \thick(P_j) \cap \sfD^\mathrm{b}(A) \\
&= \thick(P_j) \cap \left( \bigvee_{i=1}^n \thick(S_i) \right) \\
&= \bigvee_{i=1}^n \thick(P_j) \cap \thick(S_i).
\end{align*}
Of course $\RHom(P_j, S_i) = 0$ if $i\neq j$ and it follows that $\thick(P_j) \cap \thick(S_i) = 0$ if $i\neq j$. Indeed, if some non-zero $X$ were in this intersection then it would satisfy $X\in \thick(P_j)$ and hence $\RHom(P_j, X)\neq 0$ but also $X\in \thick(S_i) \subseteq \thick(P_j)^\perp$ and hence $\RHom(P_j, X)=0$ which is a difficult combination to achieve. We conclude that
\[
\thick(P_j) = \thick(P_j) \cap \thick(S_j)
\]
or in other words $P_j \in \thick(S_j)$. As above $\RHom(P_k, S_j)=0$ for $k\neq j$ and so $\RHom(P_k, P_j) = 0$ for $k\neq j$. Thus 
\[
A = \Hom(\oplus_i P_i, \oplus_i P_i) = \oplus_{i,j} \Hom(P_i, P_j) = \oplus_i \Hom(P_i,P_i)
\]
is the product of the local algebras $A_i = \Hom(P_i,P_i)$. 

The final statement follows immediately since $\sfD^\mathrm{b}(A_i) = \thick(S_i) \subseteq \sfD^\mathrm{b}(A)$ inherits distributivity from the ambient category.
\end{proof}

\begin{cor}\label{cor:dist1}
If $A$ is a finite dimensional $k$-algebra then $\Thick(\sfD^\mathrm{b}(A))$ is distributive if and only if $A$ is Morita equivalent to a product of local algebras $A_i$ such that each $\Thick(\sfD^\mathrm{b}(A_i))$ is distributive.
\end{cor}
\begin{proof}
We just proved the non-trivial direction. Conversely, if $A$ is Morita equivalent to a product of local rings $A_i$ such that each $\Thick(\sfD^\mathrm{b}(A_i))$ is distributive then
\[
\Thick(\sfD^\mathrm{b}(A)) = \Thick\left(\prod_i \sfD^\mathrm{b}(A_i)\right) = \prod_i \Thick(\sfD^\mathrm{b}(A_i))
\]
is distributive.
\end{proof}



\section{An application to modular representation theory}

In this section we fix a field $k$ of characteristic $p$ and use Theorem~\ref{thm:dist1} to characterise those finite groups $G$ for which $\Thick(\sfD^\mathrm{b}(kG))$ is a distributive lattice. In order to prove our result we need a little preparation concerning block theory.

A group $G$ is said to be $p$-nilpotent if it has a normal $p$-complement, i.e.\ $G$ is of the form $N \rtimes P$ where $N$ has order prime to $p$ and $P$ is a $p$-group (and hence a $p$-Sylow subgroup of $G$). We can understand the blocks of $kG$ as follows using \cite{CMT}*{Lemma~3.1}. If $V$ is a simple module for the semisimple algebra $kN$ with corresponding central idempotent $e$ the situtation is controlled by the inertial subgroup
\[
I_G(V) = \{ g\in G \mid V^g \cong V\}
\]
where $G$ acts on $N$ and hence on $\modu kN$ by conjugation. The idempotent $e$ gives a block $b$ of $kI_G(V)$ which is Morita equivalent to a corresponding block $B$ of $kG$. If $I_G(V) = G$ then $b = B$ is Morita equivalent to $kP$.

If $I_G(V)$ is a proper subgroup of $G$ then we need to understand the block $b$. In order to do this, it's enough to observe that $I_G(V) = N \rtimes Q$ is itself $p$-nilpotent with the same $p$-complement $N$ as we obviously have $N\unlhd I_G(V)$. Applying the same argument, for the simple $kN$-module $V$ with idempotent $e$, to $kI_G(V)$ we are in the case where the inertial group is all of $I_G(V)$ so $b$, and hence $B$, is Morita equivalent to $kQ$.

In summary, every block of $kG$ is Morita equivalent to a group algebra of a $p$-group. From this we immediately deduce the following result.

\begin{prop}\label{pnilpotent1}
If $G$ is $p$-nilpotent then $\Thick(\sfD^\mathrm{b}(kG))$ is distributive.
\end{prop}
\begin{proof}
By the discussion above $kG$ decomposes, up to Morita equivalence, into a product of group algebras $kQ_i$ where each $Q_i$ is a $p$-group. In particular, $kQ_i$ is local and so $\sfD^\mathrm{b}(kQ_i)$ is generated by its trivial module. We know $\sfD^\mathrm{b}(kQ_i)$ is a rigid tt-category and so the lattice of thick tensor ideals is a coherent frame and \emph{a fortiori} distributive. Because the trivial module, which is the unit for the tensor product, is a generator it follows that every thick subcategory is an ideal and so $\Thick(\sfD^\mathrm{b}(kQ_i))$ is distributive. The block decomposition for $kG$ gives
\[
\Thick(\sfD^\mathrm{b}(kG)) = \prod_i \Thick(\sfD^\mathrm{b}(kQ_i))
\]
and so $\Thick(\sfD^\mathrm{b}(kG))$ is distributive as claimed.
\end{proof}

We can now easily prove the following lattice theoretic characterisation of $p$-nilpotence.

\begin{thm}\label{pnilpotent2}
Let $G$ be a finite group and $k$ a field of characteristic $p$. Then the lattice $\Thick(\sfD^\mathrm{b}(kG))$ is distributive if and only if $G$ is $p$-nilpotent.
\end{thm}
\begin{proof}
The proposition we just proved gives one direction. So suppose, on the other hand, that $\Thick(\sfD^\mathrm{b}(kG))$ is distributive. By Corollary~\ref{cor:dist1} it follows that each block of $kG$ is Morita equivalent to a local ring and so, in particular, this is true for the principal block $B_0$. It is well-known (see for instance \cite{Navarro}*{Corollary~6.13}) that $B_0$ is local if and only if $G$ is $p$-nilpotent and so we are done.
\end{proof}

It follows that if $G$ is $p$-nilpotent then the stable category
\[
\smodu kG \cong \sfD^\mathrm{b}(kG)/\sfD^\mathrm{perf}(kG)
\]
also has distributive lattice of thick subcategories. However, the converse is not true: distributivity of $\Thick(\smodu kG)$ is a weaker condition and can hold without $G$ being $p$-nilpotent.

\begin{ex}\label{ex:S3}
Let $k$ be an algebraically closed field of characteristic $3$ and consider the group algebra $kS_3$. This algebra is connected and has two simples and so by Corollary~\ref{cor:dist1} the lattice of thick subcategories of $\sfD^\mathrm{b}(kG)$ is not distributive. This is also visible from Theorem~\ref{pnilpotent2} as $S_3$ is not $3$-nilpotent.

On the other hand $kS_3$ has finite representation type and the indecomposable non-projective modules form a single orbit under taking syzygies. Thus the trivial module $k$ generates $\smodu kS_3$ and so the lattice of thick subcategories is distributive (or, if one prefers, it is distributive for the more pedestrian reason that the only thick subcategories of $\smodu kS_3$ are $0$ and itself).
\end{ex}

\begin{rem}\label{rem:obvious}
For any finite group $G$ we know that $\Thick^\otimes(\sfD^\mathrm{b}(kG))$, the lattice of thick tensor ideals, is distributive (this is in some sense the fundamental theorem of tt-geometry). So the obvious way in which $\Thick(\sfD^\mathrm{b}(kG))$ can be distributive is for every thick subcategory to be an ideal. This happens precisely when $kG$ is local i.e.\ when $G$ is a $p$-group.
\end{rem}



\section{Distributivity for the perfect complexes}

In this section we prove the analogue of Theorem~\ref{thm:dist1} for the category of perfect complexes. We get essentially the same result, \emph{mutatis mutandis}, although we have to work harder for it.

\begin{thm}\label{thm:dist2}
Let $A$ be a finite dimensional $k$-algebra. If $P$ and $Q$ are non-isomorphic indecomposable projectives then distributivity of $\Thick(\sfD^\mathrm{perf}(A))$ forces $\Hom_A(P,Q) = 0$. 
\end{thm}
\begin{proof}
Suppose, with a view toward a contradiction, there exists a non-zero $f\colon P\to Q$. Without loss of generality we may assume that $f$ lies in $J\mysetminus J^2$ by replacing $Q$ with another projective if necessary and using that $J\mysetminus J^2$ generates the radical of $A$ (we could even assume it is an arrow of the quiver). We let $C = \cone(f)$.

By definition $C\in \thick(P\oplus Q) = \thick(P) \vee \thick(Q)$ and so we learn from distributivity that
\begin{align*}
\thick(C) &= \thick(C) \cap (\thick(P) \vee \thick(Q)) \\
&= (\thick(C) \cap \thick(P)) \vee (\thick(C) \cap \thick(Q)).
\end{align*} 
We want to show this is a problem for the supposed existence of $f$. 

Let us set $S_P = P/\rad(P)$ and $S_Q = Q/\rad(Q)$. The element $f$ corresponds to an extension
\[
0 \to S_P \to M \to S_Q \to 0
\]
where $M$ is of length $2$ with a non-trivial action of $f$. We thus have $\RHom(C, M) = 0$ and so $\thick(C) \subseteq {}^\perp M$. With this in mind, if $X\in \thick(C)$ then applying $\RHom(X,-)$ to the triangle corresponding to the above extension gives a triangle
\[
\RHom(X, S_P) \to 0 \to \RHom(X, S_Q) \to \Sigma \RHom(X, S_P)
\]
which tells us that $\RHom(-, S_Q)$ and $\Sigma \RHom(-, S_P)$ are isomorphic functors on $\thick(C)$. Any object of $\thick(C)$ can be represented by a minimal complex of projectives taken from $\add(P\oplus Q)$ and so for $X\in \thick(C)$ we have $X\cong 0$ if and only if
\begin{align*}
0 &= \RHom(X, S_P\oplus S_Q) \\
&\cong (\RHom(X, S_P) \oplus \RHom(X, S_Q)) \\
&\cong (\RHom(X, S_P) \oplus \Sigma \RHom(X, S_P)) \\
&\cong (\Sigma^{-1}\RHom(X, S_Q) \oplus \RHom(X, S_Q)).
\end{align*}
The problem is we now see that if $X\in \thick(C) \cap \thick(P)$ then $\RHom(X, S_Q)=0$ and if $X\in \thick(C) \cap \thick(Q)$ then $\RHom(X, S_P)=0$ and either way $X=0$. So these intersections are both trivial and so distributivity would tell us that
\[
\thick(C) = (\thick(C) \cap \thick(P)) \vee (\thick(C) \cap \thick(Q)) = 0
\]
which is not true. Thus the lattice of thick subcategories cannot be distributive unless there are no arrows between distinct vertices.
\end{proof}

\begin{cor}\label{cor:dist2}
If $A$ is a finite dimensional $k$-algebra then $\Thick(\sfD^\mathrm{perf}(A))$ is distributive if and only if $A$ is Morita equivalent to a product of local algebras $A_i$ such that each $\Thick(\sfD^\mathrm{perf}(A_i))$ is distributive.
\end{cor}
\begin{proof}
Suppose first that $\Thick(\sfD^\mathrm{perf}(A))$ is distributive. By the theorem if $P$ and $Q$ are non-isomorphic indecomposable projectives then $\Hom(P,Q) = 0 = \Hom(Q,P)$. Picking a representative set of indecomposable projectives $P_i$ we thus obtain a Morita equivalent basic algebra of the form $\prod_i \Hom(P_i,P_i)$. In $\sfD^\mathrm{perf}(A)$ we have $\thick(P_i) \cong \sfD^\mathrm{perf}(A_i)$, where $A_i = \Hom(P_i,P_i)$, and so distributivity of $\Thick(\sfD^\mathrm{perf}(A))$ implies distributivity of each $\Thick(\sfD^\mathrm{perf}(A_i))$.

On the other hand, if $A$ is Morita equivalent to a product of local rings $A_i$ such that each $\Thick(\sfD^\mathrm{perf}(A_i))$ is distributive then
\[
\Thick(\sfD^\mathrm{perf}(A)) = \Thick\left(\prod_i \sfD^\mathrm{perf}(A_i)\right) = \prod_i \Thick(\sfD^\mathrm{perf}(A_i))
\]
is distributive.
\end{proof}

This lets us add another characterisation of $p$-nilpotence to the collective list.

\begin{cor}\label{pnilpotent3}
Let $G$ be a finite group and $k$ an algebraically closed field of characteristic $p$. The following are equivalent:
\begin{itemize}
\item[(1)] $G$ is $p$-nilpotent;
\item[(2)] $\Thick(\sfD^\mathrm{b}(kG))$ is a distributive lattice;
\item[(3)] $\Thick(\sfD^\mathrm{perf}(kG))$ is a distributive lattice.
\end{itemize}
\end{cor}
\begin{proof}
We have already shown in Theorem~\ref{pnilpotent2} that (1) and (2) are equivalent. It is clear that (2) implies (3). If $\Thick(\sfD^\mathrm{perf}(kG))$ is distributive then Corollary~\ref{cor:dist2} tells us that the principal block is local and hence, as in the argument for Theorem~\ref{pnilpotent2}, the group $G$ is $p$-nilpotent.
\end{proof}



\section{Aside I: Regularity}

It's also interesting to look at the structure of individual thick subcategories and how this reflects or controls properties of the ambient category. For example, the following result yields another (well-known) characterisation of $p$-nilpotent groups.

Let $\sfT$ be a triangulated category. We say $\sfT$ is \emph{regular} if it is strongly generated, i.e.\ we have $\sfT = \thick_n(G)$ for some $G\in \sfT$ and $n\geq 0$.

\begin{thm}\label{thm:regularsimple}
Let $A$ be a finite dimensional connected symmetric algebra over $k$ and $S$ a simple $A$-module such that $\Ext^*(S,S)$ is a right coherent graded ring and $\Ext^*(S,M)$ is a finitely presented graded right $\Ext^*(S,S)$-module for every $M\in \modu A$. Then $\thick(S)$ is regular if and only if $A$ is local.
\end{thm}
\begin{proof}
Suppose first that $\thick(S)$ is regular. Let $f$ be the idempotent corresponding to $S$ and set $e = 1-f$. Consider the localization sequence
\[
\thick(S) \to \sfD^\mathrm{b}(A) \to \sfD^\mathrm{b}(eAe)
\]
where $\thick(S)$ is identified with $\sfD_{\modu A/AeA}^\mathrm{b}(A)$ the subcategory of complexes with cohomology in $\modu A/AeA$. The coherence hypothesis puts us in the situation of \cite{GreenleesS}*{Proposition~4.7} and ensures that for $M\in \sfD^\mathrm{b}(A)$ the cohomological functor $\Hom(-,M)$ on $\thick(S)^\op$ is locally finitely presented and hence representable on $\thick(S)$ by \cite{Rouquier}*{Theorem~4.16} using regularity of $\thick(S)$. Thus the functors in the above localization sequence admit right adjoints, i.e.\ we have a (left) Bousfield localization or a semiorthogonal decomposition depending on the reader's terminological preferences.

We can dualize this sequence by taking (enhanced) functors to $\sfD^\mathrm{perf}(k)$ (so we view all categories involved as stable $\infty$-categories and take the functor category in that setting) to get a semiorthogonal decomposition
\[
\begin{tikzcd}
\thick(S)^\vee \arrow[r, hook, shift right=2]  & \sfD^\mathrm{perf}(A) \arrow[r, shift right=2] \arrow[l, shift right=2]  & \sfD^\mathrm{perf}(eAe) \arrow[l, hook', shift right=2]
\end{tikzcd}
\]
of $\sfD^\mathrm{perf}(A)$. Here 
\[
\thick(S)^\vee = \Fun^\mathrm{ex}(\thick(S), \sfD^\mathrm{perf}(k)) \cong \sfD^\mathrm{perf}(\mathrm{L}_fA)
\]
where $\mathrm{L}_fA$ is the derived localization of $A$ at $f$ which is a connective dg algebra with $\mathrm{H}^0(\mathrm{L}_fA) = A/AeA$. By the assumption that $A$ is symmetric the category $\sfD^\mathrm{perf}(A)$ is $0$-Calabi-Yau and so this semiorthogonal decomposition is actually orthogonal. Thus the localization triangle for $A$
\[
\Hom(eA, A) \otimes^\mathrm{L}_{eAe} eA \to A \to \mathrm{L}_fA \to
\]
is split: so $\mathrm{L}_fA$ is concentrated in degree $0$ and is just isomorphic to $A/AeA$ and ditto for $\Hom(eA, A) \otimes^\mathrm{L}_{eAe} eA$ which is necessarily then just $AeA$.  From this we learn, by looking at the endomorphism ring of the right module $A$ and using the splitting and orthogonality relations, that $A \cong AeA \times A/AeA$. We assumed that $A$ is connected so the only way out is if $e=0$ i.e.\ $A$ is local.

The converse, namely that if $A$ is local then $\thick(S) = \sfD^\mathrm{b}(A)$ is regular, is well-known and follows easily by using the radical filtration of a bounded complex of finite dimensional $A$-modules.
\end{proof}

\begin{ex}
This is not true without the symmetry hypothesis as the following example shows. Consider a $2$-cycle modulo the radical squared, that is take the quiver 
\[
\begin{tikzcd}
1 \arrow[r, "\alpha", bend left] & 2 \arrow[l, "\beta", bend left]
\end{tikzcd} 
\]
and form the path algebra modulo the relations $\alpha\beta$ and $\beta\alpha$. Then one can compute that $\thick(S_1) \cong \sfD^\mathrm{perf}(k[\theta])$ where $\theta$ has degree $2$ and $k[\theta]$ is formal. This category is regular (and even homotopically finitely presented).
\end{ex}

\begin{cor}\label{cor:pnilpotent}
If $G$ is a finite group and $k$ is a field of characteristic $p$ then $\thick(k)$, the thick subcategory generated by the trivial module, is regular if and only if $G$ is $p$-nilpotent. Equivalently, the ring spectrum $C^*(BG; k)$ of cochains on $BG$ is regular if and only if $G$ is $p$-nilpotent.
\end{cor}
\begin{proof}
The group algebra $kG$ is symmetric so we see that regularity of $\thick(k)$ implies that the principal block of $kG$ is local. This is equivalent to $p$-nilpotence of $G$ (as used earlier, see \cite{Navarro}*{Corollary~6.13}).
\end{proof}

\begin{rem}
The corollary is, at least in some sense, well-known and there are different proofs. See for instance \cite{CCS}*{Theorem~4.1} for a homotopy theoretic perspective.
\end{rem}



\section{Aside II: Commutative rings}

We mentioned the result of Ore that the subgroups of a finite group form a distributive lattice precisely for the cyclic groups. The analogous condition has also been studied in commutative algebra: a commutative ring is called \emph{arithmetical} if its lattice of ideals is distributive. These rings seem to have been introduced by Fuchs \cite{Fuchs} in 1949. In the case of domains these are precisely the Pr\"ufer domains and there are myriad characterizations of such rings.

So, as we were motivated by Ore, one should also be motivated by Fuchs. Let $R$ be a commutative ring. The case of the perfect complexes is interesting, but does not characterise any special class of rings. The following theorem is well-known to a certain crowd.

\begin{thm}
The lattice $\Thick(\sfD^\mathrm{perf}(R))$ is distributive.
\end{thm}
\begin{proof}
We are essentially in the situation of Remark~\ref{rem:obvious}. There is a symmetric monoidal structure compatible with the triangulated structure on $\sfD^\mathrm{perf}(R)$ with unit object $R$ and $\sfD^\mathrm{perf}(R) = \thick(R)$. It follows that the lattice of thick subcategories coincides with the lattice of radical tensor ideals and hence is distributive. (See \cite{SiraGreg} for further details and references.)
\end{proof}

Distributivity for thick subcategories of the bounded derived category on the other hand is not at all well understood (see the discussion at the start of the next section) but we know it can fail as in Example~\ref{ex:nondist}.



\section{???}

In view of Corollaries \ref{cor:dist1} and \ref{cor:dist2} it seems interesting to give a characterisation of those local finite dimensional algebras $A$ such that the perfect or bounded derived category of $A$ has distributive lattice of thick subcategories. 

\begin{qn}
Let $A$ be a local finite dimensional $k$-algebra.
\begin{itemize}
\item When is $\Thick(\sfD^\mathrm{perf}(A))$ distributive?
\item When is $\Thick(\sfD^\mathrm{b}(A))$ distributive?
\end{itemize}
\end{qn}

\begin{rem}
The results of \cite{GregDb} show that if $A$ is a local artinian commutative complete intersection ring then $\Thick(\sfD^\mathrm{b}(A))$ is distributive. In a similar vein it is true for cocommutative local Hopf algebras (e.g.\ group algebras of $p$-groups in characteristic $p$).
\end{rem}

It is clear that distributivity of $\Thick(\sfD^\mathrm{b}(A))$ implies distributivity for $\Thick(\sfD^\mathrm{perf}(A))$. We saw in Corollary~\ref{pnilpotent3} that when $A$ is a group algebra the converse also holds. This is not true in general.

\begin{ex}\label{ex:nondist}
Consider the commutative local ring $A = k[x,y]/(x,y)^2$. Then, because $A$ is commutative, the lattice of thick subcategories of $\sfD^\mathrm{perf}(A)$ is distributive; in this case the only thick subcategories are $0$ and $\sfD^\mathrm{perf}(A)$. 

We claim that the lattice of thick subcategories of $\sfD^\mathrm{b}(A)$ is not distributive. This can be deduced as follows using the analysis performed in \cite{EL}. We use that $\sfD^\mathrm{b}(A) \cong \sfD^\mathrm{perf}(k\langle x,y\rangle)$ where we view the free algebra as a formal dg algebra with $x$ and $y$ in degree $1$. For a homogeneous element $f\in k\langle x,y\rangle$ let $C_f = \cone(f)$. Then we have
\[
\thick(C_{xy}) = \thick(C_{xy}) \cap (\thick(C_x, C_y))
\]
by \cite{EL}*{Lemma~4.2(1)} but, on the other hand, 
\[
(\thick(C_{xy}) \cap \thick(C_x)) \vee (\thick(C_{xy}) \cap \thick(C_y)) = 0 \vee 0 = 0
\]
by \cite{EL}*{Proposition~4.22 and Theorem~4.23}  as the elements $x,y,$ and $xy$ are all good (see their Definition~4.10).

    %
    %
%
    %
		%
		%
%
\end{ex}

This raises another natural question.

\begin{qn}
For which rings $A$ does distributivity of $\Thick(\sfD^\mathrm{perf}(A))$ imply that $\Thick(\sfD^\mathrm{b}(A))$ is distributive?
\end{qn}

Turning now to the stable category, Example~\ref{ex:S3} also raises the question of which class of finite groups is characterised by distributivity of $\Thick(\smodu kG)$.

\begin{qn}
For which finite groups $G$ is it true that $\Thick(\smodu kG)$ is a distributive lattice?
\end{qn}

Of course we know that this holds for $p$-nilpotent groups since already $\Thick(\sfD^\mathrm{b}(kG))$ is distributive. It turns out the $S_3$ example is also typical.

\begin{prop}
If $k$ is an algebraically closed field of characteristic $p$ and $G$ has cyclic $p$-Sylow subgroups, i.e.\ $kG$ has finite representation type, then $\Thick(\smodu kG)$ is distributive.
\end{prop}
\begin{proof}
Blocks with cyclic defect group are stably equivalent to symmetric Nakayama algebras (see \cite{GR}) i.e.\ to algebras of the form $\Lambda = kZ_e/\rad(kZ_e)^{em+1}$ where $Z_e$ is the cyclic quiver with $e$ vertices and $m\geq 1$. In this case the stable Auslander-Reiten quiver is the tube $\ZZ A_{me}/\tau^e$ as a translation quiver. For any symmetric algebra we have $\tau = \Sigma^{-2}$ and so if $X \in \smodu \Lambda$ is non-zero then $\thick(X) = \smodu \Lambda$ as one easily sees from the Auslander-Reiten triangles using that $\thick(X)$ is closed under suspensions and hence also under the Auslander-Reiten translation. Thus $\Thick(\smodu kG)$ is a product of copies of the distributive lattice $\{0\leq 1\}$, one for each block with non-trivial defect group, and so is distributive as claimed.
\end{proof}

There are also examples which are neither $p$-nilpotent nor of finite representation type.

\begin{ex}
Let $k$ be an algebraically closed field of characteristic $2$ and consider the group algebra $kA_4$ of the alternating group on $4$ symbols. The algebra $kA_4$ is tame, but not finite representation type, and $A_4$ is not $2$-nilpotent. However, $\smodu kA_4$ has distributive lattice of thick subcategories. One can deduce this from the description of its Auslander-Reiten quiver as given in \cite{Erdmann}*{II.7.4}. Indeed, there is a single non-periodic component containing the simple modules. Again, using that $\Sigma^{-2} = \tau$ so every thick subcategory is closed under the Auslander-Reiten translation, it is easily seen that the trivial module generates $\smodu kA_4$ and so every thick subcategory is a radical thick tensor ideal. Distributivity then follows as the tensor ideals form a distributive lattice (the stable category is a rigid tt-category).
\end{ex}

Actually one can also deduce these examples, and more, from beautiful work of Benson, Carlson, and Robinson \cite{BCR} and Benson \cite{Benson}. By \cite{Benson}*{Theorem~1.1} if $G$ is a finite group and $k$ has characteristic $p$ then provided the centralizer of every element of order $p$ in $G$ is $p$-nilpotent we have $\thick(k) = \smodu B_0$ where $B_0$ is the principal block of $kG$. So provided $B_0 = kG$ we see that the lattice of thick subcategories of $\smodu kG$ is distributive. This applies, for instance, to $\mathrm{PSL}(2, 7)$ in characteristic $2$. 

\begin{ex}
Let us conclude with an example where $\smodu kG$ does not have distributive lattice of thick subcategories; this is inspired by \cite{BCR2}*{Example~4}. Let $G = C_3 \times S_3$ and $k$ an algebraically closed field of characteristic $3$. Then the principal block is the only block and there are two simples $k$ and $\varepsilon$. We denote by $L$ the projective cover of $k$ as a $kS_3$-module viewed as a $kG$-module. Then we have
\[
\thick(L) \cap (\thick(k) \vee \thick(\varepsilon)) = \thick(L) \cap \smodu(kG) = \thick(L)
\]
but
\[
\thick(L) \cap \thick(k) \subsetneq \thick(L) \text{ and } \thick(L) \cap \thick(\varepsilon) = 0.
\]
We sketch the argument; one finds essentially everything one needs to deduce this in \cite{BCR2}. It is easy to check that $\thick(k\oplus L)$ contains $\varepsilon$ (for instance by working over $kS_3$ cf.\ Example~\ref{ex:S3}) and so must be all of $\smodu kG$. So if $L \in \thick(k)$ it would follow that $k$ generated the principal block and it doesn't as there are modules with trivial cohomology. One can easily calculate that $\varepsilon$ and $L$ are orthogonal to one another, i.e.\ there are no maps between any of their syzygies, and so $\thick(L) \cap \thick(\varepsilon) = 0$ as claimed.
\end{ex}

There are several other directions in which one could explore; the topic is relatively new.



\end{document}